\documentclass[twoside,leqno,10pt, A4]{amsart}
\usepackage{amsfonts}
\usepackage{amsmath}
\usepackage{amscd}
\usepackage{amssymb}
\usepackage{amsthm}
\usepackage{amsrefs}
\usepackage{latexsym}
\usepackage{mathrsfs}
\usepackage{bbm}
\usepackage{enumerate}
\usepackage{color}
\begin{document}

\newtheorem{theorem}[subsection]{Theorem}
\newtheorem{proposition}[subsection]{Proposition}
\newtheorem{lemma}[subsection]{Lemma}
\newtheorem{corollary}[subsection]{Corollary}
\newtheorem{conjecture}[subsection]{Conjecture}
\newtheorem{prop}[subsection]{Proposition}
\numberwithin{equation}{section}
\newcommand{\mr}{\ensuremath{\mathbb R}}
\newcommand{\mc}{\ensuremath{\mathbb C}}
\newcommand{\dif}{\mathrm{d}}
\newcommand{\intz}{\mathbb{Z}}
\newcommand{\ratq}{\mathbb{Q}}
\newcommand{\natn}{\mathbb{N}}
\newcommand{\comc}{\mathbb{C}}
\newcommand{\rear}{\mathbb{R}}
\newcommand{\prip}{\mathbb{P}}
\newcommand{\uph}{\mathbb{H}}
\newcommand{\fief}{\mathbb{F}}
\newcommand{\majorarc}{\mathfrak{M}}
\newcommand{\minorarc}{\mathfrak{m}}
\newcommand{\sings}{\mathfrak{S}}
\newcommand{\fA}{\ensuremath{\mathfrak A}}
\newcommand{\mn}{\ensuremath{\mathbb N}}
\newcommand{\mq}{\ensuremath{\mathbb Q}}
\newcommand{\half}{\tfrac{1}{2}}
\newcommand{\f}{f\times \chi}
\newcommand{\summ}{\mathop{{\sum}^{\star}}}
\newcommand{\chiq}{\chi \bmod q}
\newcommand{\chidb}{\chi \bmod db}
\newcommand{\chid}{\chi \bmod d}
\newcommand{\sym}{\text{sym}^2}
\newcommand{\hhalf}{\tfrac{1}{2}}
\newcommand{\sumstar}{\sideset{}{^*}\sum}
\newcommand{\sumprime}{\sideset{}{'}\sum}
\newcommand{\sumprimeprime}{\sideset{}{''}\sum}
\newcommand{\shortmod}{\ensuremath{\negthickspace \negthickspace \negthickspace \pmod}}
\newcommand{\V}{V\left(\frac{nm}{q^2}\right)}
\newcommand{\sumi}{\mathop{{\sum}^{\dagger}}}
\newcommand{\mz}{\ensuremath{\mathbb Z}}
\newcommand{\leg}[2]{\left(\frac{#1}{#2}\right)}
\newcommand{\muK}{\mu_{\omega}}
\newcommand{\trace}{\mathop{\mathrm{Tr}}}

\title[Moments of central values of cubic {H}ecke {$L$}-functions of $\mq(i)$]{Moments of central values of cubic {H}ecke {$L$}-functions of $\mq(i)$}

\date{\today}
\author{Peng Gao and Liangyi Zhao}

\begin{abstract}
In this paper, we study moments of central values of cubic Hecke $L$-functions in $\mq(i)$, and establish quantitative non-vanishing result for those values.
\end{abstract}

\maketitle



\section{Introduction}
\label{sec1}

Central values of $L$-functions have been investigated extensively due to the significant arithmetic information carried by them. Among the various families of $L$-functions studied in the literature, the quadratic $L$-functions have received the most attention. For example, moments of quadratic Dirichlet $L$-functions were computed to address the non-vanishing issue of the central values. Asymptotic formulas for the first and second moments of central values of the family of quadratic Dirichlet $L$-functions were obtained by M. Jutila \cite{Jutila} with the error term for the first moment improved in \cites{DoHo, MPY, ViTa}.  K. Soundararajan \cite{sound1} proved asymptotic formulas for the second and third moments with power savings.  The error term for the third moment was improved in \cite{DGH} and \cite{Young1}. \newline

   Via the method of double Dirichlet series, moments of  various families of $L$-functions associated with characters of a fixed higher order were studied in \cites{FaHL, FHL,Diac, BFH}.  Using more classical approaches, W. Luo \cite{Luo} studied moments of cubic Hecke $L$-functions on $\mq(\omega)$ with $\omega=\frac {-1+\sqrt{3}i}{2}$ and S. Baier and M. P. Young \cite{B&Y} investiaged the moments of cubic Dirichlet $L$-functions.  Motivated by these results, we shall consider, in this paper, the moments of central values of cubic Hecke $L$-functions in the Gaussian field $\mq(i)$. \newline

   Before stating our results, we shall set up some notations.  For an arbitrary number field $E$, we use $\mathcal{O}_E$ for the ring of integers in $E$ and $U_E$ for the group of units in $\mathcal{O}_E$.  For $k \in \mathcal{O}_E$, an associate of $k$ is any element of $k' \in \mathcal{O}_E$ such that $k'/k \in U_E$. We call an element $\varpi \in \mathcal{O}_E$ a prime if $(\varpi)$ is a prime ideal in $\mathcal{O}_E$ and an element $n \in \mathcal{O}_E$ square-free if $(n)$ is square-free as an ideal in $\mathcal{O}_E$. If $E$ has class number one and $q \in \mathcal{O}_E$, a characters $\chi$ modulo the ideal $(q)$ of $E$ is a homomorphism:
\begin{align*}
  \chi: \left (\mathcal{O}_E / (q) \right )^{\times} \rightarrow S^1 :=\{ z \in \mc : \hspace{0.1in} |z|=1 \}.
\end{align*}
Moreover, $\chi$ is said to be a primitive character modulo $(q)$ if it does not factor through $\left (\mathcal{O}_E / (q') \right )^{\times}$ for any proper divisor $q'$ of $q$ (i.e. $q'$ is not an associate of $q$). In this case, we say that the conduct of $\chi$ is $(q)$. When it is clear from the context, we shall use the term modulo $q$ to mean modulo $(q)$.   $\chi$ is a principal character if it factors through $\left (\mathcal{O}_E / (1) \right )^{\times}$. When $\left (\mathcal{O}_E / (q) \right )^{\times}$ is divisible by $3$, then we can define characters of order exactly $3$ modulo $q$ and we call these cubic characters modulo $q$.  It is easy to see that primitive cubic characters modulo $q$ exist if and only if $q=\prod^n_{i=1}\varpi_i$ is a product of distinct primes $\varpi_i$ such that $N_E(\varpi_i) \equiv 1 \pmod 3$ for all $i$, where $N_E$ denotes the norm from $E$ to $\mq$. In this case, every primitive cubic character $\chi$ modulo $q$ is of the form $\chi=\prod^n_{i=1}\chi_i$, where $\chi_i$ is a cubic character modulo $\varpi_i$, but regarded as a cubic character modulo $q$ by composing $\chi_i$  with the canonical homomorphism
\begin{align*}
  \left (\mathcal{O}_E / (q) \right )^{\times}  \rightarrow \left (\mathcal{O}_E / (\varpi_i) \right )^{\times}.
\end{align*}
   Note that there are precisely two cubic characters modulo $\varpi_i$ and they are primitive. \newline

   Throughout this paper, we let $K=\mq(i)$ and $F=\mq(\zeta_{12})$, where $\zeta_{12}$ is a fixed $12$-th primitive root of unity such that $\zeta^3_{12}=i, \omega=\zeta^4_{12}$. It is well-known that $\mathcal{O}_K=\mz[i]$ and $U_K=\langle i \rangle$. Both $K$ and $F$ have class number one (see \cite[Theorem 4.10]{Wa} and the table on class numbers in \cite{Wa}). Let $\chi$ be a primitive cubic character modulo some $q \in\mathcal{O}_K$. $\chi(i)=1$ since the order of $i$ is co-prime to $3$. So $\chi$ is trivial on $U_K$ and hence can be regarded as a primitive Hecke character modulo $q$ of trivial infinite type of $K$. For $m \in \mathcal{O}_K, n \in \mathcal{O}_F, (n,6)=1$, let $\psi_{m}((n)) = \leg{m}{n}_3$ where $\leg {\cdot}{\cdot}_3$ is the cubic residue symbol in $F$ (to be defined in Section \ref{sec2.4}). Then $\psi_m$ is a cubic Hecke character of trivial infinite type of $F$. It is our goal in this paper to study the first and second moments of the Hecke $L$-functions associated with these cubic characters at the central value. Our result on the first moment is
\begin{theorem}
\label{firstmoment}
  Let $w: (0,\infty) \rightarrow \mr$ be a smooth, compactly supported function.  Then
\begin{equation*}
 \sum_{\substack{ (q), \ q \in \mathcal{O}_K \\(q,6)=1}}\;  \sumstar_{\substack{\chi \bmod{q} \\ \chi^3 = \chi_0}} L(1/2, \chi) w\leg{N_K(q)}{Q} = C_0 Q \widehat{w}(0) + O(Q^{37/38 + \varepsilon}),
\end{equation*}
where $C_0 > 0$ is a constant that can be given explicitly in terms of an Euler product (see \eqref{eq:c} below), and $\widehat{w}$ is the Fourier transform of $w$.  Here and after the asterisk $*$ on the sum over $\chi$ restricts the sum to primitive characters defined on $\mathcal{O}_K$, and $\chi_0$ denotes the principal character.
\end{theorem}

    As for the second moment, we have
\begin{theorem}
\label{secmom}
  Let $Q\ge 1$. Then we have, using the same notations as Theorem~\ref{firstmoment},
\begin{equation} \label{secondm}
\sum\limits_{\substack{(q), \ q \in \mathcal{O}_K \\ N_K(q) \le Q \\ (q, 6)=1}}\ \sumstar\limits_{\substack{\chi \bmod q\\ \chi^3=\chi_0}} \left| L(1/2+it,\chi) \right|^2 \ll Q^{11/9+\varepsilon} (1+|t|)^{1+\varepsilon}.
\end{equation}
For $m \in \mathcal{O}_K$, let
\begin{equation}
\label{eq:HeckeL}
 L(s, \psi_m) =\sum_{\substack{(n), \ n \in \mathcal{O}_F \\ (n,6)=1}}\psi_{m}(n) N(n)^{-s}.
\end{equation}
   Then
\begin{equation}
\label{eq:secondmomentHecke}
 \sumstar_{\substack{ N_K(m) \leq M}} \left| L(1/2 + it, \psi_m) \right|^2 \ll M^{3/2 + \varepsilon} (1+|t^2|)^{4/3+ \varepsilon},
\end{equation}
where $\psi_m$ is the afore-mentioned cubic Hecke character of trivial infinite type of $F$ and the astrisk $*$ attached to the sum over $m$ indicates that $m$ runs over square-free elements of $\mathcal{O}_K $.
\end{theorem}

  From Theorem \ref{firstmoment} and \ref{secmom},  we readily deduce, via a standard argument (see \cite{B&Y}), the following
\begin{corollary}
\label{coro:nonvanish}
 There exist infinitely many primitive cubic characters $\chi$ of $\mq(i)$ such that $L(1/2, \chi) \neq 0$.  More precisely, the number of such characters whose norms of conductor $\leq Q$ is $\gg Q^{7/9 - \varepsilon}$.
\end{corollary}

  One important ingredient in the proof of Theorem \ref{firstmoment} and \ref{secmom} is the following large sieve-type result with cubic characters of $K$.
\begin{theorem} \label{cubiclargesieve}
Let $(a_{m})$ be an arbitrary sequence of complex numbers indexed by the elements $m$ in $\mathcal{O}_K$. Then
\begin{equation} \label{final}
\begin{split}
\sum\limits_{\substack{(q) , \ q \in \mathcal{O}_K \\ Q<N_K(q)\le 2Q \\ (q, 6)=1}}\
& \sumstar\limits_{\substack{\chi \shortmod q\\ \chi^3=\chi_0}}  \left| \ \sumstar\limits_{\substack{M<N_K(m) \le 2M}}
a_{m} \chi(m)\right|^2 \\
\ll &
(QM)^{\varepsilon}\min\left\{Q^{5/3}+M,Q^{4/3}+Q^{1/2}M,Q^{11/9}+Q^{2/3}M \right\}  \sumstar\limits_{\substack{M<N_K(m) \le 2M}}\left| a_m
\right|^2. 
\end{split}
\end{equation}
\end{theorem}

   Our proofs of the results in this paper follow along the lines of treatment given by Baier and Young in \cite{B&Y}. However, as the group of unit $U_F$ in $\mathcal{O}_F$ is infinite, we are not able to replace the summation over ideals by the summation over elements in $\mathcal{O}_F$ to apply Poisson summation to study the left-hand side expression of \eqref{final}. This is the main reason that the bounds in Theorem~\ref{secmom} and Corollary~\ref{coro:nonvanish} are weaker than their analogues in \cite{B&Y}.

\subsection{Notations} The following notations and conventions are used throughout the paper.\\
\noindent $e(z) = \exp (2 \pi i z) = e^{2 \pi i z}$. \newline
$f =O(g)$ or $f \ll g$ means $|f| \leq cg$ for some unspecified
positive constant $c$. \newline
$\mu_{K}$ denotes the M\"obius function on $\mathcal{O}_K$. \newline
$\mu_{F}$ denotes the M\"obius function on $\mathcal{O}_F$. \newline
$\zeta_{F}(s)$ is the Dedekind zeta function for the field $F$.

\section{Preliminaries}
\label{sec 2}

In this section we provide various tools used throughout the paper.
\subsection{Notations and Facts on Number Fields}
\label{sec2.0}

Here we introduce some basic notations and facts on number fields.  For an arbitrary number field $E$, set $d_E=[E:\mq]$. We call an element of $\mathcal{O}_E$ an $E$-rational integer and simply say a rational integer for any $\mq$-rational integer. Let $E'$ be another number field such that $E' \subset E$, we use $N_{E/E'}$ and $\text{Tr}_{E/E'}$ to denote the relative norm and trace from $E$ to $E'$ respectively and use $N_E$ for $N_{E/\mq}$ and  $\text{Tr}_{E}$ for $\text{Tr}_{E/\mq}$. \newline

In the case that $E$ has class number one, we denote $(\delta_E)$ and $(D_E)$ for the different and discriminant of $E$, respectively. We shall further fix generators $\delta_E, D_E$ satisfying $N_E(\delta_E)=D_E$ (see \cite[Chap.III, Theorem 2.9]{Newkirch}) and say that $\delta_E$ ($D_E$) is the different (discriminant) of $E$. It is well-known that a rational integer ramifies in $E$ if and only if it divides $D_E$ (see \cite[Chap. III, Corollary 2.12]{Newkirch}).  $\mu_E$ stands for the M\"obius function on $\mathcal{O}_E$ and we define $\mu_E(a)=\mu_E((a))$ for every $a \in \mathcal{O}_E$.  We also define $\widetilde{e}_E(k) =e(\mathrm{Tr}_E(\frac {k}{\delta_E}))$ for any $k \in E$. \newline

    Recall that $K=\ratq(i)$ and $F=\ratq(\zeta_{12})$.  We have $U_K=\langle i \rangle$ and $U_F=\langle \zeta_{12}, \epsilon \rangle$, where $\varepsilon=\frac {1+\sqrt{3}}{1-i}$ is a fundamental unit (see \cite[p. 221]{Lemmermeyer}). We fix $\delta_K=2i$ so that $D_K=-4$ and $\delta_F=2\sqrt{3}i$ so that $D_F=144$. It follows that the only primes in $\mq$ that ramify in $F$ are $2$ and $3$, in which case we have $(2)=(1+i)^2, 3=(\sqrt{3})^2$ with $1+i$ and $\sqrt{3}$ being the only primes in $F$ lying above $2$ and $3$ respectively. \newline

     Note that $F=\mq(i, \omega)$ is Galois over $\mq$. We shall denote the Galois group of $F$ over $\mq$ by $\langle \tau \rangle \times \langle \sigma \rangle$, where $\tau$ maps $i$ to $-i$ and $\sigma$ sends $\sqrt{3}i$ to $-\sqrt{3}i$. The Galois group of $F$ over $K$ thus equals $\langle \sigma \rangle$. As $F=K\mq(\omega)$, where $K$ and $\mq(\omega)$ are linearly disjoint with relatively prime discriminants, we conclude by \cite[Proposition 17, p. 68]{La} that $\mathcal{O}_F=\mathcal{O}_K[\omega]$.  It follows that an integral basis of $\mathcal{O}_F$ over $\mz$ can be taken to be $\{ 1, i, \omega, \omega i \}$. Expressing $\sqrt{3}$ using this basis, we see that every element $n \in \mathcal{O}_F$ satisfying $n \equiv 1 \pmod 3$ can be written as $n=a+b\omega+c\sqrt{3}+d\omega \sqrt{3}$ with $a, b, c, d \in \mz$ satisfying $a-1 \equiv b \equiv c+d \equiv 0 \pmod 3$ (see \cite[exercise 7.3]{Lemmermeyer}). It is known that every ideal co-prime to $3$ in $\mathcal{O}_F$ has a generator which is $\equiv 1 \pmod {3}$ (see \cite[p. 221]{Lemmermeyer}). As such generators may not be unique, we shall now henceforth fix a generator $\equiv 1 \pmod {3}$ for every prime ideal of $\mathcal{O}_F$ together with $1$ as the generator for $\mathcal{O}_F$ itself and extend to any composite $n$ with $(N_F(n), 3)=1$ multiplicatively. We call these generators primary. In particular, we note that if we write $n=n_1n_2$ with both $n$ and $n_1$ being primary, then $n_2$ is also primary.

\subsection{Primitive cubic characters}
\label{sec2.4}
Suppose that $E$ is an arbitrary number field and let $n \in \natn$ with $n \geq 2$.  Let $\mu_n(E)=\{ \zeta \in E^{\times}: \zeta^n=1 \}$ and suppose that $\mu_n(E)$ has $n$ elements. Note that the discriminant of $x^n-1$ is divisible only by the primes dividing $n$ in $\mathcal{O}_E$. It follows that for any prime $\varpi \in \mathcal{O}_E,  (\varpi, n)=1$, the map
\begin{align*}
    \zeta \mapsto \zeta \pmod \varpi : \mu_n(E) \rightarrow \mu_n(\mathcal{O}_E/\varpi)=\{ \zeta \in (\mathcal{O}_E/\varpi)^{\times}: \zeta^n=1 \}
\end{align*}
   is bijective. For such $\varpi$,  we define the $n$-th power residue symbol in $E$ for
   $a \in \mathcal{O}_E$, $(a, \varpi)=1$ by $\leg{a}{\varpi}_{n, E} \equiv
a^{(N_E(\varpi)-1)/n} \pmod{\varpi}$, with $\leg{a}{\varpi}_n \in \mu_n(E)$. When
$\varpi | a$, we define
$\leg{a}{\varpi}_{n, E} =0$.  Then these symbols can be extended
to any composite $c$ with $(N_E(c), n)=1$ multiplicatively. We further define $\leg {\cdot}{c}_{n, E}=1$ when $c \in U_E$. For any $1 \leq l \leq n$, note that
$\leg {\cdot}{c}^l_{n, E}=$ is an $n/(n,l)$-th residue symbol. \newline

    In the remainder of this paper, we shall reserve the symbol $\leg {\cdot}{\cdot}_3$ for the cubic residue symbol in $F$. It follows from \cite[Proposition 4.2, i)]{Lemmermeyer} that for all $m, n \in \mathcal{O}_F, (n, 3)=1$,
\begin{align}
\label{cubicconj}
 \leg{m}{n}^{\sigma}_3 = \leg{m^{\sigma}}{n^{\sigma}}_3.
\end{align}

    Recall that  the cubic reciprocity law (see \cite[Theorem 7.12]{Lemmermeyer}) states
that for two integers  $m, n \in \mathcal{O}_F$ satisfying $m \equiv 1 \pmod {3}, n \equiv 1 \pmod 3$,
\begin{align*}
 \leg{m}{n}_3 = \leg{n}{m}_3.
\end{align*}
  Also, the supplement laws to the cubic reciprocity law \cite[Theorem 7.12]{Lemmermeyer} imply that for $n=a+b\omega+c\sqrt{3}+d\omega \sqrt{3} \equiv 1 \pmod 3$,
\begin{align}
\label{cubicatunits}
  \leg {i}{n}_3=1, \quad \leg {\omega}{n}_3=\omega^{(a+b-1)/3}, \quad \leg {1-\omega}{n}_3=\omega^{(1-a)/3}, \quad \leg {\epsilon}{n}_3=\omega^{-c}.
\end{align}
  From this we deduce that
\begin{align*}
  \leg {\sqrt{3}}{n}_3=\omega^{b/3}.
\end{align*}

    This implies that for any $n \equiv 1 \pmod {9}$, we have
   $\leg {\epsilon}{n}_3=\leg {\sqrt{3}}{n}_3=1$.
   It follows from this that the functions $n \rightarrow \leg {\epsilon}{n}_3, n \rightarrow \leg {\sqrt{3}}{n}_3$ can be regarded as Hecke characters $\pmod {9}$ of trivial infinite type. Similarly, one sees that the function $n \rightarrow \leg {1+i}{n}_3$ can be regarded as a Hecke character $\pmod {18}$ of trivial infinite type. \newline

   Let $\varpi \in \mathcal{O}_K$ be a prime satisfying $(
   \varpi, 6)=1$. As $\mathcal{O}_F=\mathcal{O}_K[\omega]$ and the minimal polynomial for $\omega$ over $K$ is $x^2-x+1$, it follows from \cite[Proposition 2.14]{Wa} that $\varpi$ splits in $\mathcal{O}_F$ if and only if $\leg {-3}{\varpi}_{2,K}=1$. There are two types of primes in $\mathcal{O}_K$ co-prime to $6$. One is the set of primes whose norms are congruent to $1$ modulo $4$. Let $\varpi$ be such a prime, then it follows from \cite[Proposition 4.2, ii)]{Lemmermeyer} that
\begin{align*}
 \leg{-3}{\varpi}_{2,K} = \leg{-3}{N_K(\varpi)}_{\mz},
\end{align*}
   where $\leg {\cdot}{\cdot}_{\mz}$ denotes the Jacobi symbol. We then deduce that in this case $\leg {-3}{\varpi}_{2,K}=1$ if and only if $N_K(\varpi) \equiv 1 \pmod 3$. The other type of primes in $\mathcal{O}_K$ is the set of primes that are also primes in $\mq$, i.e. those rational primes that are congruent to $3$ modulo $4$. For such a prime $p \in \mz$, it follows from  \cite[Proposition 4.2, iii)]{Lemmermeyer} gives that
\begin{align*}
 \leg{-3}{p}_{2,K} = \leg{N_K(-3)}{p}_{\mz}=1.
\end{align*}
   Hence in this case all such primes split in $\mathcal{O}_F$. As in this case $N_K(p)=p^2 \equiv 1 \pmod 3$, the discussions above give that any prime $\varpi$ in $\mathcal{O}_K$ co-prime to $6$ splits in $\mathcal{O}_F$ if and only if its norm is congruent to $1$ modulo $3$. We then deduce that if $(\varpi, 6)=1$, cubic characters modulo $\varpi$ exists if and only if $\varpi$ splits in $\mathcal{O}_F$. For such a $\varpi$ that splits in  $\mathcal{O}_F$,  we write $\varpi=\mathfrak{p}\mathfrak{p}^{\sigma}=N_{F/K}(\mathfrak{p})$ with $\mathfrak{p}$ a prime in $\mathcal{O}_F$. Then the cubic symbol $\leg {\cdot}{\mathfrak{p}}_3$ gives rise to a cubic character modulo $\varpi$ by noting that $\mathcal{O}_K / \varpi \approx  \mathcal{O}_F / \mathfrak{p}$. Note also that by \eqref{cubicconj}, we have for all $a \in \left (\mathcal{O}_K / \varpi \right )^{\times}$,
\begin{align*}
 \leg{a}{\mathfrak{p}^{\sigma}}_3 = \overline{\leg{a}{\mathfrak{p}}}_3.
\end{align*}
   Thus the cubic symbol $\leg {\cdot}{\mathfrak{p}^{\sigma}}_3$ gives rise to another cubic character modulo $\varpi$ and therefore there is a one-to-one correspondence between primitive cubic characters modulo $\varpi$ and cubic residue symbols  modulo $\mathfrak{p}$ with $N_{F/K}(\mathfrak{p})=\varpi$.  By multiplicity, we extend this one-to-one correspondence to the following classification of all primitive cubic characters of conductor $q \in \mathcal{O}_K$ co-prime to $6$:
\begin{lemma}
\label{lemma:cubicclass}
 The primitive cubic characters of conductor $(q)$ with $q \in \mathcal{O}_K$ co-prime to $6$ are of the form $\chi_n:m \rightarrow (\frac{m}{n})_3$ for some $n \in \mathcal{O}_F$, $n$ square-free and not divisible by any $K$-rational primes, such that $(N_{F/K}(n)) = (q) \subset \mathcal{O}_K$.
\end{lemma}

\subsection{The Gauss sums}

Let $E$ be a number field with class number $1$ and $\chi$ be a character modulo $q$ with $q \in \mathcal{O}_E$. For $r \in \mathcal{O}_E$, we define
the Gauss sum $g(r, \chi)$ by
\begin{align*}
   g_E(r, \chi) = \sum_{x \bmod{n}} \chi(x) \widetilde{e}_E\leg{rx}{n}.
\end{align*}

    We write $g_E(\chi)$ for $g_E(1,\chi)$. Similar to the classical case shown in \cite[Chap. 9]{Da}, we have the following relation when $\chi$ is primitive:
\begin{align}
\label{grg1}
   g_E(r, \chi) = \overline{\chi} (r) g_E(\chi).
\end{align}

   In the special case of $E=F$, we define for any positive integer $l \geq 1$, $n, r \in  \mathcal{O}_F$ and $(n,6)=1$,
\begin{align*}
 g_{3,l, F}(r,n) = \sum_{x \bmod{n}} \leg{x}{n}^l_3 \widetilde{e}_F\leg{rx}{n}.
\end{align*}
   We write $g_{3, F}(r,n)$ for $g_{3,1, F}(r,n)$, $g_{3,l,F}(n)$ for $g_{3,l,F}(1,n)$ and $g_{3,F}(n)$ for $g_{3,F}(1,n)$. \newline

  When $(l, 3)=1$,  the following well-known relation (see \cite[p. 195]{P}) holds for all $n$:
\begin{align}
\label{2.1}
   |g_{3,l,F}(n)|& =\begin{cases}
    \sqrt{N_F(n)} \qquad & \text{if $n$ is square-free}, \\
     0 \qquad & \text{otherwise}.
    \end{cases}
\end{align}

  We list some properties of $g_{3,l, F}(n)$ in the following lemma, whose proof is omitted as it is similar to that of \cite[Lemma 2.4]{G&Zhao4}.
\begin{lemma} \label{quarticGausssum}
   For any $n, n_1, n_2 \in \mathcal{O}_F$ and any prime $\mathfrak{p} \in \mathcal{O}_F$ satisfying $(\mathfrak{p}nn_1n_2, 3)=1$, we have
\begin{align}
\label{eq:gmult}
 g_{3,F}(rs,n) & = \overline{\leg{s}{n}}_3 g_{3,F}(r,n), \quad (rs,n)=1, \\
\label{2.03}
   g_{3,F}(r,n_1 n_2) &=\leg{n_2}{n_1}_3\leg{n_1}{n_2}_3g_{3,F}(r, n_1) g_{3,F}(r, n_2), \quad (n_1, n_2) = 1, \\
\label{2.04}
g_{3,F}(\mathfrak{p}^k, \mathfrak{p}^l)& =\begin{cases}
    N_F(\mathfrak{p})^kg_{3,l,F}(\mathfrak{p}) \qquad & \text{if} \qquad l= k+1, \\
      \varphi(\mathfrak{p}^l)=\#(\mathcal{O}_F/(\mathfrak{p}^l))^* \qquad & \text{if} \qquad  k \geq l, l \equiv 0 \pmod {3},\\
      0 \qquad & \text{otherwise}.
\end{cases}
\end{align}
\end{lemma}

    We note that it follows from \eqref{eq:gmult}, \eqref{2.04} and the cubic reciprocity that for primary $n_1, n_2$ satisfying $(n_1, n_2)=1$,
\begin{align}
\label{eq:gtwisted}
   g_{3,F}(r,n_1 n_2) =\overline{\leg{n_2}{n_1}}_3 g_{3,F}(r, n_1) g_{3,F}(r, n_2)=g_{3,F}(n_2r, n_1) g_{3,F}(r, n_2).
\end{align}

    When $\chi$ is a primitive cubic character modulo $q$ with $q \in \mathcal{O}_K, (q, 6)=1$ , we have
\begin{align*}
    g_K(\chi) &= \sum_{a \bmod{q}}\chi(a)e\Big ( \text{Tr}_K\Big (\frac {a}{2iq}\Big )\Big ) = \sum_{a \bmod{N_{F/K}(n)}}\leg{a}{n}_3 e\Big ( \text{Tr}_{K} \Big (\frac {a}{2iN_{F/K}(n)}\Big )\Big ),
\end{align*}
   where the last equality above follows from Lemma \ref{lemma:cubicclass} with any $n \in \mathcal{O}_F, N_{F/K}(n)=q$, $n$ square-free and not divisible by any $K$-rational primes. Now we write $a=xn^{\sigma}+x^{\sigma}n \pmod {nn^{\sigma}}$, where $x$ varies over a set of representatives in $\mathcal{O}_F$ modulo $n$.  It is easy to see that as $x$ varies $\pmod n$, $a$ varies $\pmod q$. We then deduce via \eqref{eq:gmult} that
\begin{align*}
  g_K(\chi) &= \sum_{x \bmod{n}}\leg{xn^{\sigma}}{n}_3 e\Big ( \text{Tr}_K\left(\text{Tr}_{F/K}\left (\frac {x}{2i n}\right )\right )\Big )
  = \sum_{x \bmod{n}}\leg{xn^{\sigma}}{n}_3 e\Big ( \text{Tr}_F\left (\frac {x}{2i n}\right )\Big ) \\
  &=\overline{\leg{\sqrt{3}}{n}}_3\leg{n^{\sigma}}{n}_3\sum_{x \bmod{n}}\leg{x}{n}_3 e\Big ( \text{Tr}_F\left (\frac {x}{2\sqrt{3}i n}\right )\Big )=\overline{\leg{\sqrt{3}}{n}}_3\leg{n^{\sigma}}{n}_3g_{3,F}(n).
\end{align*}

   It follows from \eqref{cubicconj} and the cubic reciprocity law that for primary $n$,
\begin{align*}
   \leg {n^{\sigma}}{n}_3=\leg {n}{n^{\sigma}}^{\sigma}_3=\leg {n^{\sigma}}{n}^{\sigma}_3.
\end{align*}
   This implies that $\leg {n^{\sigma}}{n}_3 \in K$ and hence equals $1$. We then conclude that
\begin{align}
\label{wchi}
  g_K(\chi) =\overline{\leg{\sqrt{3}}{n}}_3g_{3,F}(n).
\end{align}

\subsection{Poisson Summation}
   Our proof of Theorem \ref{cubiclargesieve} needs the following Poisson summation formula.
\begin{lemma}
\label{Poissonsum} Let $n \in \mathcal{O}_K$ and let $\chi$ be a character $\pmod {n}$. For any Schwartz class function $W$,  we have for $X>0$,
\begin{align}
\label{PoissonsumQw}
   \sum_{m \in \mathcal{O}_K}\chi(m)W\left(\frac {N_K(m)}{X}\right)=\frac {X}{N_K(n)}\sum_{k \in
   \mathcal{O}_K}g_K(k,\chi)\widetilde{W}_K\left(\sqrt{\frac {N_K(k)X}{N_K(n)}}\right),
\end{align}
   where
\begin{align*}
   \widetilde{W}_K(t) &=\int\limits^{\infty}_{-\infty}\int\limits^{\infty}_{-\infty}W(N_K(x+yi))e\left(- t\text{Tr}_K(\frac {x+yi}{2i})\right)\dif x \dif y, \quad t
   \geq 0.
\end{align*}
\end{lemma}

    The above lemma is established in \cite[Lemma 2.11]{G&Zhao4} and we note that it is also shown in \cite{G&Zhao4} that $\widetilde{W}_K(t)\in
    \mr$ for any $t \geq 0$ and that $\widetilde{W}_K$ is a Schwartz class function as well.

\subsection{The approximate functional equation}
Suppose that $E=K$ or $F$ and $q \in \mathcal{O}_E$. Let $\chi$ be a Hecke character of trivial infinite type of $E$ modulo $q$.
The Hecke $L$-function associated with $\chi$ is defined $\Re(s) > 1$ by
\begin{equation*}
  L(s, \chi_q) = \sum_{0 \neq \mathcal{A} \subset
  \mathcal{O}_E}\chi(\mathcal{A})(N_E(\mathcal{A}))^{-s}, 
\end{equation*}
  where $\mathcal{A}$ runs over all non-zero integral ideals in $E$ and $N(\mathcal{A})$ is the
norm of $\mathcal{A}$.  E. Hecke showed that $L(s, \chi)$ admits
analytic continuation to an entire function and satisfies a
functional equation \cite[Corollary 8.6]{Newkirch}:
\begin{equation*}
  \Lambda(s, \chi_q) = g_{E}(\chi)(N_E(q))^{-1/2}\Lambda(1-s, \overline{\chi}_q),
\end{equation*}
   where (note that $E$ is totally imaginary)
\begin{equation*}
  \Lambda(s, \chi_q) = (|D_E|N_E(q))^{s/2}\Gamma_E(s)L(s, \chi_q) \quad \mbox{and} \quad \Gamma_E(s)=((2\pi)^{-s}\Gamma(s))^{d_E/2}.
\end{equation*}
   We refer the reader to \cite{Newkirch} for a more detailed discussion of the Hecke characters and the associated $L$-functions.  \newline

    Let $G(s)$ be any even function which is holomorphic and bounded in the strip $-4<\Re(s)<4$ satisfying $G(0)=1$.  Using an approximate functional equation (see Theorem 5.3 of \cite{HIEK}), we have (see \cite[Section 2.3]{G&Zhao3}) the following expression for $L(1/2+it, \chi_q)$ if $\chi_q$ is primitive:
\begin{equation} \label{approxfuneq}
\begin{split}
 L \left( \frac{1}{2}+it, \chi_q \right) = \sum_{0 \neq \mathcal{A} \subset
  \mathcal{O}_E} & \frac{\chi_q(\mathcal{A})}{N_E(\mathcal{A})^{1/2+it}}V_{t, E} \left(\frac{N_E(\mathcal{A})}{A} \right) \\
  & + \frac{g_{E}(\chi_q)}{N_E(q)^{1/2}}\left(\frac {1}{|D_E|N_E(q)} \right )^{it} \frac {\Gamma_E (1/2-it)}{\Gamma_E (1/2+it)}\sum_{0 \neq \mathcal{A} \subset
  \mathcal{O}_E}\frac{\overline{\chi}_q(\mathcal{A})}{N_E(\mathcal{A})^{1/2-it}}V_{-t, E}\left(\frac{
  N_E(\mathcal{A})}{B} \right),
     \end{split}
\end{equation}
    where $A, B$ are positive real numbers satisfying $AB=N_E(q)$,
\begin{align*}
  V_{t,E} \left(\xi \right)=\frac {1}{2\pi
   i}\int\limits\limits_{(2)} \gamma_{t,E}(s) G(s)\frac
   {\xi^{-s}}{s} \ \dif s \quad \mbox{with} \quad\gamma_{t,E}(s)=\frac {\Gamma_E(s+1/2+it)}{\Gamma_E(1/2+it)}(\sqrt{|D_E|})^{s}.
\end{align*}

    We write $V_E$ for $V_{0,E}$ and note that for a suitable $G(s)$ (for example $G(s)=e^{s^2}$) and $c>0$ (see \cite[Proposition 5.4]{HIEK}):
\begin{align}
\label{2.15}
  V_{t,E} \left(\xi \right) \ll \left( 1+\frac{\xi}{(1+|t|)^{d_E/2}} \right)^{-c}.
\end{align}

   On the other hand, if $G(s)=1$, the $j$-th derivative of $V_E(\xi)$ satisfies (see \cite[Lemma 2.1]{sound1}) the bounds
\begin{equation} \label{2.07}
      V_E\left (\xi \right) = 1+O(\xi^{1/2-\epsilon}) \; \mbox{for} \; 0<\xi<1   \quad \mbox{and} \quad V^{(j)}_E\left (\xi \right) \ll \exp\left( -\frac {d_E}{2}\xi^{2/d_E} \right) \; \mbox{for} \; \xi >0, \ j \geq 0.
\end{equation}

\subsection{Analytic behavior of Dirichlet series associated with Gauss sums}
\label{section: smooth Gauss}
  For any Hecke character $\lambda \pmod {18}$ of trivial infinite type of $F$, we define
\begin{align}
\label{h}
   h(r,s;\lambda)=\sum_{\substack{n \text{ primary} \\ (n,r)=1}}\frac {\lambda(n)g_{3,F}(r,n)}{N_F(n)^s}.
\end{align}

   We omit the proof of the next lemma as it is similar to that of Lemma 3.6 in \cite{B&Y}. Note that in the next lemma we use the convention that all sums over elements of $\mathcal{O}_F$ are restricted to primary elements.
\begin{lemma}
\label{lemma:laundrylist}
 Suppose $f$, $\alpha$ are square-free and $(r,f) = 1, (rf\alpha, 3)=1$, and set
\begin{equation*}
 h(r,f,s;\lambda) = \sum_{(n,rf) = 1} \frac{\lambda(n) g_{3,F}(r,n)}{N_F(n)^s}, \quad h_{\alpha}(r,s;\lambda) = \sum_{(n,\alpha) =1} \frac{\lambda(n) g_{3,F}(r,n)}{N_F(n)^s}.
\end{equation*}
Furthermore suppose that $r = r_1 r_2^2 r_3^3$ where $r_1 r_2$ is square-free, and let $r_3^*$ be the product of primes dividing $r_3$.
Then
\begin{align*}
 &h(r,f,s;\lambda) = \sum_{a | f} \frac{\mu_{F}(a) \lambda(a) g_{3,F}(r,a)}{N_F(a)^{s}} h(ar,s;\lambda), \quad  h(r_1 r_2^2 r_3^3, s;\lambda) = h(r_1r_2^2, r_3^*,s;\lambda), \\
&h(r_1 r_2^2, s;\lambda) = \prod_{\substack{p| r_2 \\ p \text{ prime in } \mathcal{O}_F}} (1 - \lambda(p)^3 N_F(p)^{2-3s})^{-1}
h_{r_1}(r_1 r_2^2, s;\lambda), \\
& h_{r_1}(r_1 r_2^2, s;\lambda) = \prod_{\substack{p| r_1 \\ p \text{ prime in } \mathcal{O}_F}}(1 -\lambda(p)^3 N_F(p)^{2-3s})^{-1} \sum_{a | r_1} \mu_{F}(a) N_F(a)^{1-2s} \lambda(a)^2 \overline{g_{3,F}}(r_1 r_2^2/a, a) h_1(r_1r_2^2/a, s;\lambda).
\end{align*}
\end{lemma}

    Similar to the proof of \cite[Lemma 3.5]{B&Y}, we deduce readily from Lemma \ref{lemma:laundrylist} and a result of S. J. Patterson  \cite[Lemma, p. 200]{P} the following result concerning the analytic behavior of $h(r,s;\lambda)$ on $\Re(s) >1$.
\begin{lemma}{\cite[Lemma 2.5]{G&Zhao1}}
\label{lem1} The function $h(r,s;\lambda)$ has meromorphic continuation to the entire complex plane. It is holomorphic in the
region $\sigma=\Re(s) > 1$ except possibly for a pole at $s = 4/3$. For any $\varepsilon>0$, letting $\sigma_1 = 3/2+\varepsilon$, then for $\sigma_1 \geq \sigma \geq \sigma_1-1/2$, $|s-4/3|>1/6$, we have
\[ h(r,s;\lambda) \ll N_F(r)^{(\sigma_1-\sigma+\varepsilon)/2}(1+t^2)^{ \sigma_1-\sigma+\varepsilon}, \]
  where $t=\Im(s)$. Moreover, suppose that $r = r_1 r_2^2 r_3^3$ where $r_1 r_2$ is square-free, then the residue satisfies
\[ \mathrm{Res}_{s=4/3}h(r,s;\lambda) \ll N_F(r)^{\varepsilon}N_F(r_1)^{-1/6+\varepsilon}. \]
\end{lemma}

\subsection{The large sieve with cubic residue symbols}  An important tool in the proof of our results is the following large sieve inequality for cubic residue symbols in $F$, which is a special case of \cite[Theorem 1.3]{BGL}:
\begin{lemma} \label{cubicls}
Let $M,N$ be positive integers, and let $(a_{n})$ be an arbitrary sequence of complex numbers indexed by primary elements $n$ in $\mathcal{O}_F$. Then we have for any $\varepsilon > 0$,
\begin{equation*}
 \sumstar_{\substack{m \text{ primary} \\ \\N_F(m) \leq M}} \left| \ \sumstar_{\substack{n \text{ primary} \\N_F(n) \leq N}} a_{n} \leg {m}{n}_3 \right|^2
 \ll_{\varepsilon} \left( M + N + (MN)^{2/3} \right)(MN)^{\varepsilon} \sum_{\substack{n \text{ primary} \\N_F(n) \leq N}} |a_{n}|^2,
\end{equation*}
where the asterisks indicate that $m$ and $n$ run over square-free elements of $\mathcal{O}_F$ and $(\frac
{\cdot}{n})_3$ is the cubic residue symbol.
\end{lemma}

\section{The cubic large sieve in $\mq(i)$}
\label{sec8}

In this section we shall establish Theorem \ref{cubiclargesieve}.  Using Lemma \ref{lemma:cubicclass}, we see that
\begin{equation} \label{trans}
\begin{split}
\sum\limits_{\substack{ (q), \ q \in \mathcal{O}_K \\ Q<N_K(q)\le 2Q\\ (q,6)=1}}\
\sumstar\limits_{\substack{\chi \bmod q\\ \chi^3=\chi_0}} \left| \ \sumstar\limits_{\substack{M<N_K(m)\le 2M}} a_{m} \chi(m)\right|^2 &=
\sumprime\limits_{\substack{(n) , \ n\in \mathcal{O}_F \\ Q<N_F(n)\le 2Q\\
(n,6)=1}} \left| \
\sumstar\limits_{\substack{M< N_K(m) \le 2M}} a_{m} \chi_n(m)\right|^2,
\end{split}
\end{equation}
where the prime indicates that $n$ is square-free and has no
$K$-rational prime divisor.

\subsection{Definition of certain norms}
To bound the expression on the right-hand side of \eqref{trans}, we begin by defining a norm corresponding
to that double sum.  Set
\begin{equation*}
B_1(Q,M)=\sup\limits_{(a_{m})} \| a_{m} \|^{-2}
\sumprime\limits_{\substack{(n) , \ n\in \mathcal{O}_F \\ Q<N_F(n)\le 2Q\\
(n,6)=1}} \left| \
\sumstar\limits_{\substack{M< N_K(m) \le 2M}} a_{m} \chi_n(m)\right|^2, \quad \mbox{where} \quad 
\|a_{m} \|^2 = \sum_{m} |a_{m}|^2. 
\end{equation*}
Without any loss of generality, we may assume that $(a_{m})$ is not the zero sequence. With applications to Theorem \ref{secmom} in mind, we further define a norm $B_2(Q,M)$ in the same way as $B_1(Q,M)$ except removing the condition that $n$ has no $K$-rational prime divisor. \newline

   By the duality principle, we have
\begin{equation}
\label{B1C1}
B_1(Q,M)=C_1(M,Q),
\end{equation}
   where we define the norm $C_1(M,Q)$ dual to $B_1(Q,M)$ by
\begin{equation*}
C_1(M,Q)=\sup\limits_{(b_{(n)})} \| b_{(n)} \|^{-2}
\sumstar\limits_{\substack{M< N_K(m) \le 2M}}  \left| \
\sumprime\limits_{\substack{(n) , \ n\in \mathcal{O}_F \\ Q<N_F(n)\le 2Q\\
(n,6)=1 }} b_{(n)} \chi_n(m)\right|^2,
\end{equation*}
   where $(b_{(n)})$ are non-zero sequences of complex numbers indexed by the ideals $(n)$ in $\mathcal{O}_F$. \newline

Furthermore, we define a norm $C_2(M,Q)$ by extending the summation in the definition of $C_1(M,Q)$ to all elements $m \in \mathcal{O}_K$ with $M < N_K(m) \leq 2M$. Trivially, we have
\begin{equation} \label{C12}
C_1(M,Q)\le C_2(M,Q).
\end{equation}

\subsection{Comparison of the norms} \label{section:comparison}

We need the following lemma on the norms defined in the previous section.

\begin{lemma} \label{normlemma} Let $Q,M\ge 1$ and $v$ a fixed positive integer . Then we have for any $\varepsilon>0$,
\begin{align}
\label{C2e1}
C_2(M,Q) & \ll (QM)^{\varepsilon}\left(M +Q^{5/3}\right), \\
\label{C22}
C_2(M,Q) & \ll M^{\varepsilon}Q^{1-1/v}\sum\limits_{j=0}^{v-1}C_2(2^jM^v,Q)^{1/v}, \\
\label{B11}
B_1(Q_1,M) & \ll B_1(Q_2,M), \quad \text{ if $Q_1,M\ge 1$ and $Q_2\ge CQ_1\log (2Q_1M)$}, \\
\label{B21}
B_2(Q,M) & \ll (\log{2Q})^3 Q^{1/2}X^{-1/2} B_1(XQ^{\varepsilon},M), \quad \text{for some $X$ with $1\le X\le Q$}.
\end{align}
\end{lemma}
\begin{proof}
  As the proofs of \eqref{C22}-\eqref{B21} are similar to those of \cite[(31), (33)]{B&Y}, we omit them here by only pointing out that the proof of \eqref{B21} makes use of \eqref{B11}. To prove \eqref{C2e1},
  we recall that $C_2(M,Q)$
is the norm of the sum
\begin{equation} \label{eq:C2defW}
\sum\limits_{\substack{M< N_K(m) \le 2M}}  \left| \ \sumprime\limits_{\substack{(n), \ n\in \mathcal{O}_F \\ Q<N_F(n)\le 2Q\\
(n,6)=1}} b_{(n)} \chi_n(m)\right|^2  \ll \sum\limits_{m\in \mathcal{O}_K }
W\left(\frac{N_K(m)}{M}\right) \left| \ \sumprime\limits_{\substack{(n) , n\in \mathcal{O}_F \\ Q<N_F(n)\le 2Q\\
(n,6)=1 }} b_{(n)}
\chi_n(m)\right|^2,
\end{equation}
where $W:\mr \rightarrow \mr$ is a fixed smooth, nonnegative, compactly-supported function such that $W(x) \geq 1$ for $1 \leq x \leq 2$. \newline

Expanding out the sum on the right-hand side of \eqref{eq:C2defW}, we get
\begin{equation*}
\sumprime \limits_{\substack{(n_1), (n_2), \ n_1,n_2\in  \mathcal{O}_F \\
Q<N_F(n_1), N_F(n_2)\le 2Q\\ (n_1n_2, 6)=1}}
b_{(n_1)}\overline{b}_{(n_2)} \sum\limits_{m \in \mathcal{O}_K }
W\left(\frac{N_K(m)}{M}\right) \chi_{n_1}\overline{\chi}_{n_2}(m).
\end{equation*}
Now we extract the greatest common divisor $\Delta$ of $n_1$ and
$n_2$, getting
\begin{equation*}
\sumprime_{\substack{\Delta \in \mathcal{O}_F \\ \\ N_F(\Delta) \leq 2Q \\ (\Delta, 6)=1 \\ \Delta \text{ primary}}}\ \sumprime \limits_{\substack{(n_1), (n_2), \ n_1,n_2\in
\mathcal{O}_F \\ \frac{Q}{N_F(\Delta)}< N_F(n_1), N_F(n_2)\le
\frac{2Q}{N_F(\Delta)} \\
(n_1,n_2)=1 \\ (n_1 n_2, 6N_{F/K}(\Delta) = 1 )}} b_{(n_1
\Delta)}\overline{b}_{(n_2 \Delta)} \sum\limits_{\substack{m\in \mathcal{O}_K
\\ (m, N_{F/K}(\Delta)) = 1}} W\left(\frac{N_K(m)}{M}\right)
\chi_{n_1}\overline{\chi}_{n_2}(m).
\end{equation*}
  We further write $\delta = (n_1, n^{\sigma}_2)$ and change variables
via $n_1 \rightarrow \delta n_1$, $n_2 \rightarrow
\delta^{\sigma} n_2$ to get
\begin{equation} \label{eq:prePoisson}
\begin{split}
\sumprime_{\substack{\Delta \in \mathcal{O}_F \\ N_F(\Delta) \leq 2Q \\ (\Delta, 6)=1 \\
\Delta \text{ primary}}}\
\sumprime_{\substack{\delta \in \mathcal{O}_F \\ N_F(\delta)
\leq \frac{2Q}{N_F(\Delta)} \\ (\delta, 6)=1, \ \delta \text{ primary} \\
(N_{F/K}(\delta), N_{F/K}(\Delta)) = 1}} & \
\sumprime\limits_{\substack{(n_1), (n_2), \ n_1,n_2\in \mathcal{O}_F\\
\frac{Q}{N_F(\delta \Delta)}<N_F(n_1),N_F(n_2)\le \frac{2Q}{N_F(\delta
\Delta)} \\ (N_{F/K}(n_1),
N_{F/K}(n_2)) = 1 \\ (N_{F/K}(n_1 n_2), 6N_{F/K}(\delta \Delta)) = 1 } } b_{(n_1 \Delta \delta)}\overline{b}_{(n_2
\Delta \delta^{\sigma})}
\\
& \times \sum\limits_{\substack{m\in \mathcal{O}_K\\ (m, N_{F/K}(\Delta)) = 1}}
W\left(\frac{N_K(m)}{M}\right)
\chi_{n_1}\overline{\chi}_{n_2 \delta}(m),
\end{split}
\end{equation}
where we use that $\chi_{\delta}\overline{\chi_{\delta^{\sigma}}}=\chi_{\delta}^2=
\overline{\chi_{\delta}}$. Next we remove the coprimality
condition in the sum over $m$ by the M\"obius function. We shall do so by fixing a generator for every prime ideal of $\mathcal{O}_K$ together with $1$ as the generator for $\mathcal{O}_K$ and extend to any composite $0 \neq n \in \mathcal{O}_K$ multiplicatively. We call these generators primary in $\mathcal{O}_K$. Thus we get
\begin{equation*}
\sum\limits_{\substack{m\in \mathcal{O}_K\\ (m, N_{F/K}(\Delta)) = 1}}
W\left(\frac{N_K(m)}{M}\right)
\chi_{n_1}\overline{\chi}_{n_2 \delta}(m) =
\sum_{\substack{l \text{ primary in } \mathcal{O}_K \\ \ell | N_{F/K}(\Delta)}} \mu_K(\ell)
\chi_{n_1}\overline{\chi}_{n_2\delta}(\ell)
\sum\limits_{\substack{m\in \mathcal{O}_K}} W\left(\frac{N_K(m)}{M/N_K(\ell)}\right)
\chi_{n_1}\overline{\chi}_{n_2\delta}(m),
\end{equation*}
which by the Poisson summation formula \eqref{PoissonsumQw} is
\begin{equation} \label{afterpoisson}
M \sum_{\substack{l \text{ primary in } \mathcal{O}_K\\ \ell | N_{F/K}(\Delta)}} \frac{\mu_K(\ell)
\chi_{n_1}\overline{\chi}_{n_2\delta}(\ell)}{N_K(\ell) N_{F}(n_1 n_2 \delta)} \sum\limits_{\substack{h\in \mathcal{O}_K}}
g_K(h,\chi_{n_1}\overline{\chi}_{n_2\delta})\widetilde{W}_K\left(\frac{MN_K(h)}{N_K(\ell) N_{F}(n_1 n_2 \delta)}\right).
\end{equation}

   As $\chi_{n_1}\overline{\chi}_{n_2\delta}$ is a primitive character modulo $N_{F/K}(n_1n_2\delta)$, we then derive from \eqref{grg1} and \eqref{wchi} that for any $h \in \mathcal{O}_K$,
\begin{align*}
  g_K(h,\chi_{n_1}\overline{\chi}_{n_2\delta}) = \overline{\chi}_{n_1}\chi_{n_2\delta}(h)g_K(\chi_{n_1}\overline{\chi}_{n_2\delta})= \overline{\chi}_{n_1}\chi_{n_2\delta}(h)\overline{\chi_{n_1}\overline{\chi}_{n_2\delta}} \left( \sqrt{3} \right)g_{3, F}(n_1n^{\sigma}_2\delta^{\sigma}) .
\end{align*}
  Using this, we can rewrite \eqref{afterpoisson} as
\[
M \sum_{\substack{l \text{ primary in } \mathcal{O}_K\\ \ell | N_{F/K}(\Delta)}} \frac{\mu_K(\ell)
\chi_{n_1}\overline{\chi}_{n_2\delta}(\ell)}{N_K(\ell) N_{F}(n_1 n_2 \delta)} \overline{\chi_{n_1}\overline{\chi}_{n_2\delta}}\left( \sqrt{3} \right) g_{3, F}(n_1n^{\sigma}_2\delta^{\sigma})\sum\limits_{\substack{h\in \mathcal{O}_K}}
\overline{\chi}_{n_1}\chi_{n_2\delta}(h)\widetilde{W}_K\left(\frac{MN_K(h)}{N_K(\ell) N_{F}(n_1 n_2 \delta)}\right).
\]

   When $h=0$, the expression in \eqref{afterpoisson} vanishes unless
$n_1 = n_2 = \delta =1$. Hence, the contribution of $h=0$ to
\eqref{eq:prePoisson} is
\begin{equation*}
\ll MQ^{\varepsilon} \sumprime_{\substack{\Delta \in \mathcal{O}_F, \ (\Delta, 6)=1 \\ Q<N_F(\Delta) \leq 2Q \\
\Delta \text{ primary} }} |b_{(\Delta)}|^2 \ll M
Q^{\varepsilon}\| b_{(n)} \|^2.
\end{equation*}

   In the sequel, we assume that $n_1, n_2$ are primary. We then note that it follows from \eqref{grg1} and \eqref{eq:gtwisted} that
\begin{align*}
  g_{3, F}(n_1n^{\sigma}_2\delta^{\sigma})=\leg {n_2\delta}{n^{\sigma}_1}_3 \leg {\delta}{n_2}_3g_{3,F}(n_1)g_{3,F}(n^{\sigma}_2)g_{3,F}(\delta^{\sigma}).
\end{align*}
   Using this and changing $n_1 \rightarrow n^{\sigma}_1$, we see that the contribution of
$h\not=0$ to the sum in \eqref{afterpoisson} takes the form
\[ S_W(M,Q) = M\sumprime_{\substack{\Delta \in \mathcal{O}_F \\ N_F(\Delta) \leq 2Q \\ (\Delta, 6)=1 \\
\Delta \text{ primary}}}\
\sumprime_{\substack{\delta \in \mathcal{O}_F,  \ N_F(\delta)
\leq \frac{2Q}{N_F(\Delta)} \\ (\delta, 6)=1, \ \delta \text{ primary} \\
(N_{F/K}(\delta), N_{F/K}(\Delta)) = 1}}\frac{g_{3,F}(\delta^{\sigma})\chi_{\delta}(\sqrt{3})}{ N_{F}(\delta)} \sum_{\substack{l \text{ primary in } \mathcal{O}_K\\ \ell | N_{F/K}(\Delta)}}\frac{\mu_K(\ell)}{N_K(\ell)} \overline{\chi}_{\delta}(\ell)  \sum\limits_{\substack{h\in \mathcal{O}_K \\ h\not=0}} \chi_{\delta}(h) U(\Delta,\delta,l,h) ,\]
where
\begin{equation*}
U(\Delta,\delta,l,h)=\sumprime\limits_{\substack{n_1,n_2\in \mathcal{O}_F , \ n_1,n_2 \text{ primary} \\
\frac{Q}{N_F(\delta \Delta)}<N_F(n_1),N_F(n_2)\le \frac{2Q}{N_F(\delta
\Delta)} \\ (N_{F/K}(n_1),
N_{F/K}(n_2)) = 1 \\ (N_{F/K}(n_1 n_2), 6N_{F/K}(\delta \Delta)) = 1 } }\widetilde{W}_K\left(\frac{M N_K(h)}{N_K(\ell) N_{F}(n_1 n_2 \delta)}\right) c_{\Delta,\delta,\ell,h,n_1}
c_{\Delta,\delta,\ell,h,n_2}'
\left(\frac{n_1}{n_2}\right)_3,
\end{equation*}

\[
c_{\Delta,\delta,\ell,h,n}:=\chi_{n^{\sigma}}(\ell )\chi_{n}\left( \sqrt{3}\delta h \right)
\frac{g_{3,F}(n^{\sigma})}{N_{F}(n)}
b_{n^{\sigma}\Delta\delta} \; \; \; \mbox{and} \; \; \;
c_{\Delta,\delta,\ell,h,n}':= \chi_{n^{\sigma}}(\ell )\chi_{n}\left( \sqrt{3}\delta h \right)
\frac{g_{3,F}(n^{\sigma})}{N_{F}(n)}
\overline{b}_{n\Delta \delta^{\sigma}}.
\]

  Using \eqref{2.1}, we see that the coefficients $c,c'$ satisfy the bounds
\begin{equation*}
c_{\Delta,\delta,l,h,n}, \; c_{\Delta,\delta,l,h,n}' \ll  \left(\frac{N_F(\delta\Delta)}{Q}\right)^{1/2} \left| b_{n\Delta \delta^{\sigma}} \right|.
\end{equation*}

  To estimate $S_W(M,Q)$, observe that we may freely truncate the sum over $h$ for
\begin{equation*}
 N_K(h) \le \frac{Q^2 N_K(\ell)}{N_F(\delta
\Delta^2)M} (QM)^{\varepsilon}=:H
\end{equation*}
since $\widetilde{W}_K$ has rapid decay. More precisely, if we let $S_W
(M,Q) = S'_W(M,Q) + E$ where $S'_W(M,Q)$ is the contribution to
$S_W(M,Q)$ from $0 < N_K(h) \leq H$, then $E \ll
(MQ)^{-100} \| b \|^2$. Using this and \eqref{2.1}, we then arrive at the following bound
\begin{equation} \label{hnot0}
S_W''(M,Q) \ll
M \sumprime_{\Delta} \sumprime_{\delta} \frac{1}{N_F(\delta)^{1/2}} \sum_{l | N_{F/K}(\Delta)} \frac{1}{N_K(l)} \sum_{0 < N_K(h) \leq H} \left| U(\Delta,\delta,l,h) \right|,
\end{equation}

  To estimate $U(\Delta,\delta,l,h)$, we note the following consequence of Lemma \ref{cubicls} (see \cite[(43)]{B&Y}): Suppose that $d_n$, $d_n'$ are arbitrary complex numbers supported on square-free primary $n \in \mathcal{O}_F$, with $N_F(n) \leq X$.  Then by Cauchy's inequality,
\begin{equation} \label{eq:HBcubic2}
\left| \sum_{m, n} d_m d_n' \leg{m}{n}_3 \right|  \leq \left( \sum_m |d_m|^2 \right)^{1/2} \left(\sum_{m} \left| \sum_n d_n' \leg{n}{m}_3 \right|^2 \right)^{1/2} \ll X^{2/3 + \varepsilon} \|d_m \| \cdot \|d_n' \|.
\end{equation}

  After using M\"{o}bius inversion to remove the coprimality condition $\left( N_{F/K}(n_1),N_{F/K}(n_2) \right)=1$ and the Mellin inversion formula to remove the weight function $\widetilde{W}_K$, we apply \eqref{eq:HBcubic2} to see that
\begin{equation*}
U(\Delta,\delta,l,h) \ll (QM)^{\varepsilon}\left(\frac{N_F(\delta \Delta)}{Q}\right)^{1/3}
\sumprime_{n} |b_{n}|^2.
\end{equation*}
  Inserting this into \eqref{hnot0} and summing trivially over all the other variables gives \eqref{C2e1}.
\end{proof}

\subsection{Completion of the proof of Theorem \ref{cubiclargesieve}}
    We note that it follows from \eqref{B1C1} -- \eqref{C22} that we have
\begin{equation}
\label{C2egen}
B_1(Q,M) \ll (QM)^{\varepsilon}\left(Q^{1-1/v} M +
Q^{1+2/(3v)}\right)
\end{equation}
for any $v\in \mn$. Theorem \ref{cubiclargesieve} now follows by taking $v=1$, $2$ and $3$.

\section{Proof of Theorem \ref{secmom}}
\label{section:secondmomentproof}

  In this section, we establish Theorem \ref{secmom} and a variant of it. Since these results are easy consequences of Theorem \ref{cubiclargesieve} and all steps are standard, we will only sketch the arguments. \newline

We shall first prove \eqref{secondm}.  Using the approximate functional equation \eqref{approxfuneq}, with $E=K, A=B=N_K(q)^{1/2}$ and Cauchy's inequality, we estimate the second moment in question by
\begin{equation*}
\sum\limits_{\substack{(q), \ q \in \mathcal{O}_K \\ N_K(q) \le Q \\ (q, 6)=1}} \ \sumstar\limits_{\substack{\chi \shortmod q\\ \chi^3=\chi_0}} \left| L(1/2+it,\chi) \right|^2 \leq 2 \sum\limits_{\substack{(q) , \ q \in \mathcal{O}_K \\ N_K(q) \le Q \\ (q, 6)=1}} \ \sumstar\limits_{\substack{\chi \shortmod q\\ \chi^3=\chi_0}} \left| \sum_{\substack{ m \text{ primary in } \mathcal{O}_K }} \frac{\chi(m)}{N_{K}(m)^{1/2 + it}} V_{it,K}\left(\frac{N_K(m)}{\sqrt{N_K(q)}}\right)
\right|^2.
\end{equation*}
 In view of \eqref{2.15}, we may truncate the sum over $m$ so that $N_K(m) \leq M:=(Q(1+|t|)^2)^{1/2 + \varepsilon}$ with a negligibly small error. \newline

Then we break the summations over $q$ and $m$ into dyadic intervals and remove the weight $V_{it}$ using the Mellin transform. We further write $m=d^2n$, where $n$ is square-free in $\mathcal{O}_K$, and use the Cauchy-Schwarz inequality again. Eventually, we arrive at sums of the form
\begin{equation*} \label{end}
\sum\limits_{\substack { d \text{ primary in } \mathcal{O}_K  \\ N_{K}(d)\le \sqrt{2M}}} \frac{1}{N_{K}(d)} \sum\limits_{\substack{(q), \ q \in \mathcal{O}_K \\ Q<N_K(q) \le 2 Q \\ (q,6)=1}}\ \sumstar\limits_{\substack{\chi \shortmod q\\ \chi^3=\chi_0}} \left| \ \sumstar_{\substack{m \text{ primary in } \mathcal{O}_K \\ M/N_{K}(d)^2<N_{K}(m)\le 2M/N_{K}(d)^2}} \frac{\chi(m)}{N_{K}(m)^{1/2 + it}} \right|^2
\end{equation*}
which we then estimate using Theorem \ref{cubiclargesieve}.  More precisely, we use
\eqref{final} with the last term, $Q^{11/9}+Q^{2/3}M$, in the minimum.  Plugging this bound in and summing trivially over $d$ gives \eqref{secondm}. \newline

Next we establish \eqref{eq:HeckeL}. It follows from \eqref{cubicatunits} that for $m$ square-free in $\mathcal{O}_K$, the character $\psi_m(n) = \leg{m}{n}_3$ can be regarded as a Hecke character of trivial infinite type with conductor $\mathfrak{f}$ satisfying $m/(3,m) | \mathfrak{f}$, $\mathfrak{f} | 9m$.  Thus the Hecke $L$-function $L(s, \psi_m)$, viewed as a degree $4$ $L$-function over $\mq$, has analytic conductor $\ll N_F(m) (1+ t^2)^2 = N_K(m)^2(1+t^2)^2$.  A variant on the above argument reduces the problem of estimating \eqref{eq:secondmomentHecke} to bounding
\begin{equation*}
 \sumstar_{N_K(m) \leq M} \left| \ \sumstar_{\substack{ n \text{ primary}, \ (n,6)=1 \\ N_F(n) \leq c_1(M(1+t^2))^{1+\varepsilon}}} \frac{\chi_n(m)}{N_F(n)^{1/2 + it}} \right|^2,
\end{equation*}
where $c_1$ is some positive constant and the star attached to the sum over $n$ indicates that the sum runs over square-free elements $n \in \mathcal{O}_F$.  We further note that it follows from \eqref{B21} and \eqref{C2egen} that we have the bound $B_2(Q,M) \ll (QM)^{\varepsilon}(Q^{4/3} + Q^{1/2} M)$. This together with the duality principle then gives the desired estimate. \newline

  We also establish the following variant of \eqref{eq:HeckeL}, which will be needed in the proof of Theorem \ref{firstmoment}. Let $d \in \mathcal{O}_K$ with $N_K(d) \ll Q$, then
\begin{equation}
\label{eq:Heckevariant}
 \sum_{N_K(m) \leq M} \frac{|L(1/2 + it, \psi_{m(6d)^3})|}{\sqrt{N_K(m)}} \ll Q^{\varepsilon}M^{3/4 + \varepsilon} (1+|t|)^{4/3 + \varepsilon}.
\end{equation}
To prove this version, we factor $m$ as $m_1 m_2^2 m_3^3$ where $m_1m_2$ is square-free in $\mathcal{O}_K$.  Then $\psi_{m(6d)^3}$ equals $\psi_{m_1} \overline{\psi}_{m_2}$ times a principal character.  For each fixed $m_2$, we then generalize \eqref{eq:HeckeL} to give
\begin{equation*}
 \sumstar_{\substack{N_K(m_1) \leq M_1 \\ (m_1, m_2)=1}} \left| L(1/2 + it, \psi_{m_1} \overline{\psi}_{m_2} \right|^2 \ll M_1^{3/2+\varepsilon} m_2^{4/3+\varepsilon} (1+|t^2|)^{4/3+\varepsilon}.
\end{equation*}
With this bound and Cauchy's inequality, it is easy to sum over $m_2$ trivially, giving \eqref{eq:Heckevariant}.

\section{Proof of Theorem \ref{firstmoment}}
\label{sec3}
  We let
\begin{equation*}
\mathcal{M} :=  \sum_{\substack{ (q), \ q \in \mathcal{O}_K \\(q,6)=1}}\;  \sumstar_{\substack{\chi \bmod{q} \\ \chi^3 = \chi_0}} L(1/2, \chi) w\leg{N_K(q)}{Q}.
\end{equation*}

   We note that by Lemma \ref{lemma:cubicclass}, we have
\begin{equation*}
\mathcal{M} = \sumprime_{\substack{n \text{ primary} \\ (n, 6)=1}} L(1/2, \chi_n) w\leg{N_K(N_{F/K}(n))}{Q}= \sumprime_{\substack{n \text{ primary} \\ (n, 6)=1}} L(1/2, \chi_n) w\leg{N_F(n)}{Q},
\end{equation*}
   where the prime indicates that the sum runs over square-free elements $n \in \mathcal{O}_F$ that have no $K$-rational prime divisors. \newline

   Applying the approximate functional equation \eqref{approxfuneq} with $G(s)=1, A_n B =N_F(n)$ and writing $V$ for $V_F$,  gives $\mathcal{M} = \mathcal{M}_1 + \mathcal{M}_2$, where
\begin{align*}
 \mathcal{M}_1 &= \sumprime_{\substack{n \text{ primary} \\ (n, 6)=1}} \sum_{\substack { (m) \\ 0 \neq m \in
  \mathcal{O}_K}} \frac{\chi_n(m)}{\sqrt{N_K(m)}}V\leg{N_K(m)}{A_n} w\left(\frac{N_F(n)}{Q}\right), \; \mbox{and} \\
 \mathcal{M}_2 &= \sumprime_{\substack{n \text{ primary} \\ (n, 6)=1}} \sum_{\substack { (m) \\ 0 \neq m \in
  \mathcal{O}_K}} \frac{g_{K}(\chi)}{N_K(n)^{1/2}} \frac{\overline{\chi}_n(m)}{\sqrt{N_K(m)}} V\leg{N_K(m)}{B} w\left(\frac{N_F(n)}{Q}\right).
\end{align*}

    It remains to evaluate $\mathcal{M}_1$ and bound $\mathcal{M}_2$.  The results are summarized with
\begin{lemma}
 We have
\begin{equation}
\label{eq:M1estimate}
 \mathcal{M}_1 =  C_0 Q \widetilde{w}(1) + O(Q^{1/2 + \varepsilon} A^{3/4+ \varepsilon} + Q A^{-1/6 + \varepsilon}),
\end{equation}
and
\begin{equation}
\label{eq:M2bound}
 \mathcal{M}_2 \ll Q^{5/6 + \varepsilon} B^{1/6 + \varepsilon}+Q^{2/3 + \varepsilon} B^{5/6 + \varepsilon}.
\end{equation}
\end{lemma}
   Choosing $B = Q^{7/19}$, whence $A = Q^{12/19}$  gives Theorem \ref{firstmoment}.  The constant $C_0$ is given more explicitly in \eqref{eq:c}.
\subsection{Evaluating $\mathcal{M}_1$, the main term}
\label{section:M1}
  We detect the condition that a primary $n \in \mathcal{O}_F$ has no $K$-rational prime divisor using the formula
\begin{equation}
\label{eq:ratmob}
 \sum_{\substack{d \text{ primary in } \mathcal{O}_K \\ d |n}} \mu_{K}(d) =
\begin{cases}
 1, \quad \text{$n$ has no $K$-rational prime divisor}, \\
 0, \quad \text{otherwise}.
\end{cases}
\end{equation}
   We apply this formula and for fixed $d \in \mathcal{O}_K$, we fix a unit $\eta_d \in U_F$ such that $d\eta_d$ is primary (in $\mathcal{O}_F$).  We then make a change of variable $n \rightarrow d\eta_d n$ so that $n$ is also primary.  Further note that as $n$ is co-prime to $6$, any divisor $d$ of $n$ is also co-prime to $6$.  Hence $d$ is not divisible by any primes in $\mathcal{O}_K$ that ramifies in $\mathcal{O}_F$. Thus if $d$ is square-free in $\mathcal{O}_K$, it is also square-free as an element of $\mathcal{O}_F$. The condition that $d\eta_d n$ is square-free then simply means that $n$ is square-free and $(d,n) = 1$. We further define $A$ by $AB=Q$ so that $A_n = A \frac{N_F(n)}{Q} \asymp A$ for all $n$ under consideration, in view of the support of $w$. Thus
\begin{align*}
 \mathcal{M}_1 = \sum_{\substack{d \text{ primary in } \mathcal{O}_K \\ (d,6)=1 }} \mu_{K}(d) \sum_{\substack { (m) \\ 0 \neq m \in
  \mathcal{O}_K}}  \frac{\leg{m}{d}_3}{\sqrt{N_K(m)}}  \ \sumstar_{\substack{n  \text{ primary}  \\ (n,6d) = 1}} \leg{m}{n}_3 V\left(\frac{N_K(m)}{A} \frac{Q}{N_F(nd)} \right) w\left(\frac{N_F(nd)}{Q}\right).
\end{align*}

  Using M\"{o}bius inversion to detect the condition that $n$ is square-free, we get
\begin{equation*}
 \mathcal{M}_{1} = \sum_{\substack{d \text{ primary in } \mathcal{O}_K \\ (d,6)=1 }} \mu_{K}(d) \sum_{\substack{ l \text{ primary} \\ (l,6d) = 1}} \mu_{F}(l) \sum_{\substack { (m) \\ 0 \neq m \in
  \mathcal{O}_K}}  \frac{\leg{m}{dl^2}_3}{\sqrt{N_K(m)}} \mathcal{M}_1(d,l,m),
\end{equation*}
where
\[
 \mathcal{M}_1(d,l,m) = \sum_{\substack{n  \text{ primary}  \\ (n,6d) = 1}} \leg{m}{n}_3 V\left(\frac{N_K(m)}{A} \frac{Q}{N(ndl^2)} \right) w\left(\frac{N_F(ndl^2)}{Q}\right).
\]
Next we use the Mellin transform of the weight function to express the sum over $n$ as a contour integral involving the Hecke $L$-function.  By Mellin inversion,
\[ 
 V\left(\frac{N_K(m)}{A} \frac{Q}{N_F(ndl^2)} \right) w\left(\frac{N_F(ndl^2)}{Q}\right) = \frac{1}{2 \pi i} \int\limits_{(2)} \leg{Q}{N_F(ndl^2)}^s \widetilde{f}(s) \dif s,
\; \mbox{where} \;
\widetilde{f}(s) = \int\limits_0^{\infty} V\left(\frac{N_K(m)}{A} x^{-1} \right) w(x) x^s \frac{\dif x}{x}.
\]
Integration by parts and using \eqref{2.07} shows $\widetilde{f}(s)$ is a function satisfying the bound for all $\Re (s) >0$, $R>0$,
\begin{equation*}
 \widetilde{f}(s) \ll (1 + |s|)^{-R} (1 + N_K(m)/A)^{-R}.
\end{equation*}

   We then have
\begin{equation*}
 \mathcal{M}_1(d,l,m) = \frac{1}{2 \pi i} \int\limits_{(2)} \leg{Q}{N_F(dl^2)}^s L(s, \psi_{m(6d)^3}) \widetilde{f}(s) ds,
\end{equation*}
   where $L\left( s,\psi_{m(6d)^3} \right)$ is defined as in \eqref{eq:HeckeL}. \newline

We estimate $\mathcal{M}_{1}$  by moving the contour to the half line.  When $m$ is a cubic the Hecke $L$-function has a pole at $s=1$.  We set $\mathcal{M}_0$ to be the contribution to $\mathcal{M}_{1}$ of these residues, and $\mathcal{M}_1'$ to be the remainder. \newline

  We evaluate $\mathcal{M}_0$ first. Note that
\begin{equation*}
 \mathcal{M}_0 =  \sum_{\substack{d \text{ primary in } \mathcal{O}_K \\ (d,6)=1 }} \mu_{K}(d) \sum_{\substack{ l \text{ primary} \\ (l,6d) = 1}} \mu_{F}(l) \sum_{\substack { (m) \\ 0 \neq m \in
  \mathcal{O}_K}}  \frac{\leg{m}{dl^2}_3}{\sqrt{N_K(m)}} \frac{Q}{N_F(dl^2)} \widetilde{f}(1) \text{Res}_{s=1} L(s, \psi_{m(6d)^3}),
\end{equation*}
where using the Mellin convolution formula shows
\begin{equation} \label{3.4}
 \widetilde{f}(1) = \int\limits_0^{\infty} V\left(\frac{N_K(m)}{A} x^{-1} \right) w(x) \dif x = \frac{1}{2 \pi i} \int\limits_{(2)} \leg{A}{N_K(m)}^s \widetilde{w}(1+s) \gamma_{0,F}(s) \frac{\dif s}{s}, \; \; \mbox{with} \; \; \widetilde{w}(s) = \int\limits_0^{\infty} w(x) x^{s-1} \dif x.
\end{equation}

   From the discussion in Section \ref{sec2.4}, it is not difficult to see that $\psi_{m(6d)^3}$ is a principal character only if $m$ is a cubic, in which case
\begin{equation*}
 L \left( s, \psi_{m(6d)^3} \right) = \zeta_{F}(s) \prod_{\mathfrak{p} | 6dm} \left( 1 - N_F(\mathfrak{p} \right)^{-s}),
\end{equation*}
   where $\mathfrak{p}$ runs over primes in $\mathcal{O}_F$. \newline

  Let $C_1$ be the residue of $\zeta_{F}(s)$ at $s=1$, then
\begin{equation*}
 \mathcal{M}_0 =  C_1 Q  \sum_{\substack { (m) \\ 0 \neq m \in
  \mathcal{O}_K}}\frac{\widetilde{f}(1)}{N_K(m)^{3/2}} \prod_{ \mathfrak{p}| 6m} \left( 1- N_F(\mathfrak{p})^{-1} \right)  \sum_{\substack{d \text{ primary in } \mathcal{O}_K \\ (d,6m)=1 }}\frac{ \mu_{K}(d) }{N^2_K(d)}\prod_{\mathfrak{p} | d} \left( 1 - N_F(\mathfrak{p})^{-1} \right)  \sum_{\substack{ l \text{ primary} \\ (l,6md) = 1}} \frac{\mu_{F}(l)}{N_F(l^2)}.
\end{equation*}

   Computing the sum over $l$ explicitly, we obtain
\begin{align*}
  \mathcal{M}_0= &\frac{C_1 Q}{\zeta_{F}(2)} \sum_{\substack { (m) \\ 0 \neq m \in
  \mathcal{O}_K}}\frac{\widetilde{f}(1)}{N_K(m)^{3/2}} \prod_{ \mathfrak{p}| 6m} \left( 1- N_F(\mathfrak{p})^{-1} \right) \sum_{\substack{d \text{ primary in } \mathcal{O}_K \\ (d,6m)=1 }} \frac{ \mu_{K}(d) }{N^2_K(d)}\prod_{\mathfrak{p} | d} \left( 1 - N_F(\mathfrak{p})^{-1} \right)   \prod_{\mathfrak{p} | 6md} \left( 1- N_F(\mathfrak{p})^{-2} \right)^{-1} \\
 =& \frac{C_1 Q}{\zeta_{F}(2) } \sum_{\substack { (m) \\ 0 \neq m \in
  \mathcal{O}_K}}\frac{\widetilde{f}(1)}{N_K(m)^{3/2}}\prod_{\mathfrak{p} | 6m} \left( 1+ N_F(\mathfrak{p})^{-1} \right)^{-1}  \sum_{\substack{d \text{ primary in } \mathcal{O}_K \\ (d,6m)=1 }}\frac{ \mu_{K}(d) }{N^2_K(d)}\prod_{\mathfrak{p} | d} \left(1 + N_F(\mathfrak{p})^{-1} \right)^{-1}.
\end{align*}

   We define
\begin{align*}
  C_2= \sum_{\substack{d \text{ primary in } \mathcal{O}_K \\ (d,6)=1 }} \frac{ \mu_{K}(d) }{N^2_K(d)}\prod_{\mathfrak{p} | d} \left( 1 + N_F(\mathfrak{p})^{-1} \right)^{-1}.
\end{align*}
$C_2$ is clearly a constant. Using this, we arrive at
\begin{align*}
 \mathcal{M}_0 = \frac{C_1C_2 Q}{\zeta_{F}(2)} \sum_{\substack { (m) \\ 0 \neq m \in
  \mathcal{O}_K}}\frac{\widetilde{f}(1)}{N_K(m)^{3/2}}\prod_{\mathfrak{p} | 6m} \left(1+ N_F(\mathfrak{p})^{-1} \right)^{-1}
\prod_{\varpi | m/(m,6)} \left( 1-N_K(\varpi)^{-2} \prod_{\mathfrak{p} | \varpi} \left(1 + N_F(\mathfrak{p})^{-1} \right)^{-1} \right)^{-1},
\end{align*}
  where $\varpi$ runs over primes in $\mathcal{O}_K$. Let
\begin{equation*}
 C(u) = \sum_{\substack { (m) \\ 0 \neq m \in
  \mathcal{O}_K}}N_K(m)^{-u} \prod_{\mathfrak{p} | 6m} \left( 1+ N_F(\mathfrak{p})^{-1} \right)^{-1}
 \prod_{\varpi | m/(m,6)} \left(1-N_K(\varpi)^{-2} \prod_{\mathfrak{p} | \varpi} \left( 1 + N_F(\mathfrak{p} \right)^{-1})^{-1} \right)^{-1},
\end{equation*}
which is holomorphic and bounded for $\Re (u) \geq 1 + \delta > 1$.  Then
\begin{equation*}
 \mathcal{M}_0 = \frac{C_1C_2 Q}{ \zeta_{F}(2) } \frac{1}{2 \pi i} \int\limits_{(2)} A^s C(3/2 + 3s) \widetilde{w}(1+s)\gamma_{0,F}(s)  \frac{\dif s}{s} .
\end{equation*}

  We move the contour of integration to $-1/6 + \varepsilon$, crossing a pole at $s=0$ only.  The new contour contributes $O(A^{-1/6 + \varepsilon} Q)$, while the pole at $s=0$ gives
\begin{equation}
\label{eq:c}
  C_0 Q \widetilde{w}(1), \quad \text{where} \quad C_0= C_1C_2 \zeta^{-1}_{F}(2)C(3/2) .
\end{equation}
Note that $C(u)$ converges absolutely at $u=3/2$ so it is easy to express $C(3/2)$ explicitly as an Euler product, if desired.
We then conclude that
\begin{align}
\label{m0}
 \mathcal{M}_0 = C_0 Q  \widetilde{w}(1)+O(QA^{-1/6 + \varepsilon} ).
\end{align}

\subsection{Evaluating $\mathcal{M}_1$, the remainder term}
\label{section:remainderterm}

 In this section, we estimate $\mathcal{M}'_1$. By bounding everything with absolute values, we see that for any $R>0$,
\begin{align*}
  \mathcal{M}_1' \ll  \sum_{\substack{d \text{ primary in } \mathcal{O}_K \\ N_K(d) \ll \sqrt{Q}}}\frac{1}{\sqrt{N_K(d^2)}} \sum_{\substack{l \text{ primary} \\ N_F(l) \ll \sqrt{Q}}} \frac{1}{\sqrt{N_F(l^2)}}  \sum_{\substack { (m) \\ 0 \neq m \in
  \mathcal{O}_K}} \frac{\sqrt{Q}}{\sqrt{N_K(m)}} & \left( 1+N_K(m)/A \right)^{-R} \\
  & \times \int\limits_0^{\infty} \left| L \left( 1/2 + it, \psi_{m(6d)^3} \right) \right| (1+|t|)^{-R} \dif t.
\end{align*}

   Apply the estimation from \eqref{eq:Heckevariant} and note that we can assume $N(m) \ll A^{1+\varepsilon}$ here, we
obtain
\begin{equation*}
\mathcal{M}_1' \ll Q^{1/2 + \varepsilon} A^{3/4+ \varepsilon}.
\end{equation*}
  This combined, with \eqref{m0}, gives \eqref{eq:M1estimate}.

\subsection{Estimating $\mathcal{M}_2$}
\label{section:M2}

Using \eqref{wchi}, we have
\begin{align*}
\label{eq:M2recall}
 \mathcal{M}_{2} &=  \sum_{\substack { (m) \\ 0 \neq m \in
  \mathcal{O}_K}} \frac{1}{\sqrt{N_K(m)}} V\leg{N_K(m)}{B} \sumprime_{\substack{n  \text{ primary}  \\ (n,6) = 1}} \frac{\overline{\chi}_n \left( \sqrt{3}m \right) g_{3,F}(n)}{ \sqrt{N_F(n)}}  w\left(\frac{N_F(n)}{Q}\right).
\end{align*}
  Now we need the following result:
\begin{lemma}
\label{lemma:sumofGauss}
 For any $m \in \mathcal{O}_K$, write $m = m_0 m_1$ where $m_0$ is a unit times a power of $1+i$ and a power of $3$ while $(m_1,6)=1$, $m_1$ primary in  $\mathcal{O}_K$. We further write $m_1=m_2m^2_3m^3_4$ with $m_2m_3$ being square-free in $\mathcal{O}_K$. Then we have
\begin{equation*}
\label{eq:sumofGausssums}
H'(m,Q):= \sumprime_{\substack{n  \text{ primary}  \\ (n,6) = 1}}\frac{\overline{\chi}_n \left( \sqrt{3}m \right)g_{3,F}(n)}{N_F(n)^{1/2}}  w\left(\frac{N_F(n)}{Q}\right) \ll  Q^{2/3 + \varepsilon} N_K(m)^{1/3+\varepsilon}+ Q^{5/6 + \varepsilon}N_K(m)^{\varepsilon} N_K(m_2)^{-1/3+\varepsilon}.
\end{equation*}
\end{lemma}

Our first move in the proof of Lemma \ref{lemma:sumofGauss} is to use
M\"{o}bius inversion given in \eqref{eq:ratmob}, to remove the condition that $n$ has no $K$-rational prime divisor.  For fixed $d \in \mathcal{O}_K, (d, 6)=1$, we fix a unit $\eta_d \in U_F$ such that $d\eta_d$ is primary (in $\mathcal{O}_F$), we then make a change of variable $n \rightarrow d\eta_d n$ so that $n$ is also primary. It follows from  \eqref{eq:gtwisted} that
\begin{align*}
 g_{3,F}(d\eta_dn) = \overline{\leg{d\eta_d}{n}}_3 g_{3,F}(d\eta_d) g_{3,F}(n).
\end{align*}

  We use the notation $\widetilde{g}_{3,F}(d) = g_{3,F}(d\eta_d) N_F(d)^{-1/2}$ so that $|\widetilde{g}_{3,F}(d)| \leq 1$ by \eqref{2.1}. It follows further from \eqref{2.1} that $g_{3,F}(n) = 0$ unless $n$ is square-free. This gives
\begin{equation*}
\label{eq:M2sieved}
 H'(m,Q) =  \sum_{\substack{d \text{ primary in } \mathcal{O}_K \\ (d,6)=1 }}  \mu_{K}(d) \tilde{g}_{3,F}(d)\overline{\chi}_d \left( \sqrt{3}m \right) H \left( \sqrt{3} d \eta_d m, Q/N_K(d)^2 \right), 
\end{equation*}
where
\begin{equation*}
H \left( \sqrt{3} d \eta_d m, X \right) =  \sum_{\substack{n \text{ primary }\\ (n, 6)=1}}\overline{\chi}_{n} \left( \sqrt{3} d \eta_d m \right)  \frac {g_{3,F}(n)}{N_F(n)^{1/2}} w\left(\frac{N_F(n)}{X}\right).
\end{equation*}
We estimate $H$ with the following
\begin{lemma}
\label{lemma:Hbound}
 For any $l \in \mathcal{O}_F$, write $l = l_0 l_1$ where $l_0$ is a unit times a power of $1+i$ and a power of $\sqrt{3}$ while $l_1$ is primary. We further write $l_1=l_2l^2_3l^3_4$ with $l_2l_3$ square-free.  Then we have
\begin{equation*}
\label{eq:Hbound}
 H(l, X) \ll X^{1/2 + \varepsilon} N_F(l_1)^{1/4+\varepsilon} + X^{5/6}N_F(l_1)^{\varepsilon} N_F(l_2)^{-1/6}.
\end{equation*}
\end{lemma}
\begin{proof}
 Writing $l = l_0 l_1$ as above, we see that
\begin{equation*}
\overline{\chi}_n(l) = \overline{\leg{l}{n}}_3 = \overline{\leg {l_1}{n}_3} \cdot \overline{\leg{l_0}{n}}_3.
\end{equation*}
From the discussion in Section \ref{sec2.4}, the function $\lambda(n) = \overline{\leg{l_06^3}{n}}_3$ is a Hecke character $\pmod{18}$ of trivial infinite type.
Thus
\begin{equation*}
H(l,X) =\sum_{\substack\substack{n \text{ primary }}}  \frac{\lambda(n) \overline{\leg {l_1}{n}}_3g_{3,F}(n)}{ \sqrt{N_F(n)}} w\left(\frac{N_F(n)}{X}\right).
\end{equation*}
  Note that the identity \eqref{eq:gmult} implies $\overline{\leg {l_1}{n}}_3 g_{3,F}(n) = g_{3, F}(l_1, n)$ for $(n, l_1) = 1$.  Introducing the Mellin transform $\widetilde{w}(s)$ of $w$ given in \eqref{3.4}, we get
\begin{equation}
\label{eq:HlX}
 H(l,X) =
\frac{1}{2 \pi i} \int\limits_{(2)} \widetilde{w}(s) X^s h( l_1, 1/2 + s;\lambda) \dif s,
\end{equation}
   where $h$ is defined in \eqref{h}. \newline

  We move the line of integration in \eqref{eq:HlX} to $\Re (s) = 1/2 + \varepsilon$, crossing a pole at $s =5/6$, which by Lemma \ref{lem1} contributes
\begin{equation*}
\ll X^{5/6}N_F(l_1)^{\varepsilon} N_F(l_2)^{-1/6}.
\end{equation*}
  Also Lemma \ref{lem1} implies the contribution from the new line of integration gives
\begin{equation*}
\ll  X^{1/2+\varepsilon} N_F(l_1)^{1/4+\varepsilon}.
\end{equation*}
This completes the proof of Lemma \ref{lemma:Hbound}.
\end{proof}

  Now, to prove Lemma \ref{lemma:sumofGauss}, we treat $N_K(d) \leq Y$ and $N_K(d) > Y$ separately, where $Y$ is a parameter to be chosen.  For $N_K(d) \leq Y$ we use Lemma \ref{lemma:Hbound}, while for $N_K(d) > Y$ we use the trivial bound $H(l,X) \ll X$.  Thus, writing  write $dm = (dm)_0 (dm)_1$ where $(dm)_0$ is a unit times a power of $1+i$ and a power of $\sqrt{3}$ while $(dm)_1$ is primary. We further write $(dm)_1=(dm)_2(dm)^2_3(dm)^3_4$ with $(dm)_2(dm)_3$ square-free in $\mathcal{O}_F$, we have
\begin{align*}
 H'(m,Q)  \ll  \sum_{0 \neq N_K(d) \leq Y} & \leg{Q}{N_K(d^2)}^{1/2+\varepsilon} N_K((dm)^2)^{1/4+\epsilon} \\
 & + \sum_{\substack{ N_K(d) \leq Y \\ (d,6)=1}} |\mu_K(d)|\leg{Q}{N_K(d)^2}^{5/6} N_K(dm)^{\varepsilon}N_F((dm)_2)^{-1/6} + \sum_{N_K(d) > Y} \frac{Q}{N_K(d)^2}.
\end{align*}
  Note that for square-free $d \in \mathcal{O}_K, (d,6)=1$, $d$ is also square-free in $\mathcal{O}_F$. On writing $m = m_0 m_1$, $m_1=m_2m^2_3m^3_4$  where $m_0$ is a unit in $\mathcal{O}_K$ times a power of $1+i$ and a power of $3$, $m_1$ primary in  $\mathcal{O}_K$ satisfying $(m_1,6)=1$ with $m_2m_3$ being square-free in $\mathcal{O}_K$. It is then easy to see that for square-free $d \in \mathcal{O}_K$ with $(d,6)=1$, we have $N_K((dm)_2)^{-1/6} \leq N_K(d)^{1/6}N_K(m_2)^{-1/6}$. Using this, we see that
\begin{equation*}
 H'(m,Q) \ll Q^{1/2 + \varepsilon} Y^{1/2} N_K(m)^{1/2+\varepsilon} + Q Y^{-1} + Q^{5/6} N_K(m)^{\varepsilon}N_K(m_2)^{-1/3 }.
\end{equation*}
 Choosing $Y = Q^{1/3} N_K(m)^{-1/3}$ gives Lemma \ref{lemma:sumofGauss}.  By summing trivially over $m$ one easily deduces \eqref{eq:M2bound}.  \newline

\noindent{\bf Acknowledgments.} P. G. is supported in part by NSFC grant 11871082 and L. Z. by the FRG grant PS43707 and the Faculty Silverstar Award PS49334.
Parts of this work were done when P. G. visited the University of New South Wales (UNSW) in August 2018. He wishes to thank UNSW for the invitation, financial support and warm hospitality during his pleasant stay.

\bibliography{biblio}
\bibliographystyle{amsxport}

\vspace*{.5cm}

\noindent\begin{tabular}{p{8cm}p{8cm}}
School of Mathematical Sciences & School of Mathematics and Statistics \\
Beihang University & University of New South Wales \\
Beijing 100191 China & Sydney NSW 2052 Australia \\
Email: {\tt penggao@buaa.edu.cn} & Email: {\tt l.zhao@unsw.edu.au} \\
\end{tabular}

\end{document}